\documentclass[11pt,draft,leqno]{amsart}
  \usepackage{amsmath}
\usepackage{amssymb}
\usepackage{amsfonts}
\usepackage{eucal}
\usepackage{color}
 \makeatletter
     \def\section{\@startsection{section}{1}%
     \z@{.7\linespacing\@plus\linespacing}{.5\linespacing}%
     {\bfseries
     \centering
     }}
     \def\@secnumfont{\bfseries}
     \makeatother
\usepackage[full]{textcomp}

\newtheorem{theorem}{Theorem}[section]
\newtheorem{lemma}[theorem]{Lemma}

\newtheorem{corollary}[theorem]{Corollary}
\theoremstyle{definition}

\theoremstyle{remark}
\newtheorem{remark}[theorem]{{\bf  Remark}}
\numberwithin{equation}{section}
    \usepackage{a4wide}
\font\sevenrm =cmr10 at  7  pt
\def\ddate {\sevenrm \ifcase\month\or January\or
February\or March\or April\or May\or June\or July\or
August\or September\or October\or November\or December\fi\! {\the\day}, \!{\sevenrm\the\year}}

\def \a{{\alpha}}
\def \b{{\beta}}

\def \d{{\delta}}
\def \e{{\varepsilon}}

\def \O{{\Omega}}
\def \p{{\varphi}}
\def \t{{\vartheta}}

\def \m{{\mu}}
\def \s{{\sigma}}

\def \A{{\mathcal A}}

\def \C{{\mathcal C}}

\def \N{{\bf N}}

\def \P{{\bf P}}

\def \qq{{\qquad}}
\def \R{{\bf R}}

\def \Z{{\bf Z}}
\def \dd{{\rm  d }}
   
\def \noi{{\noindent}}

\def\E{{\mathbb E}}

\def\P{{\mathbb P}}

\def\R{{\mathbb R}}

\def\Z{{\mathbb Z}}

\def\N{{\mathbb N}}

\def\C{{\mathbb C}}  
  
\font\phh=cmcsc10
at  8 pt
 at 10,4pt

  at 9,7pt
   \font\phh=cmr10 at  8,2 pt
  \scrollmode
 at 9,3  pt
 at 8,2  pt

   \title[\rm 
  A uniform semi-local limit theorem along    sets of multiples for i.i.d.  sums]
  {A uniform semi-local limit theorem along    sets of multiples   for sums of i.i.d. 
   random variables}
  \author{ Michel  J.\ G. Weber 
}
\address{ Michel Weber: IRMA, Universit\'e
Louis-Pasteur et C.N.R.S.,   7  rue Ren\'e Descartes, 67084   
Strasbourg Cedex, France. }
\email{michel.weber@math.unistra.fr ; m.j.g.weber@mailo.com
}

\begin{document}
 \maketitle

\renewcommand{\thefootnote}{} {{
\footnote{2010 \emph{Mathematics Subject Classification}:  {\phh Primary: 60F15, 60G50 ;
Secondary: 60F05}. 
}
\footnote{\emph{Key words and phrases}: semi-local limit theorem, divisors, Bernoulli random variables, i.i.d. sums, Theta functions, distribution, prime numbers.}
 \renewcommand{\thefootnote}{\arabic{footnote}}
\setcounter{footnote}{0}
  \begin{abstract}   Let $X  $  be  a square integrable random variable  with basic probability space $(\O, \A, \P)$,   taking values in a  lattice $\mathcal L(v_0,1)=\big\{v_k=v_0+ k,k\in \Z\big\}$ and such that $\t_X   =\sum_{k\in \Z}\P\{X=v_k\}\wedge \P\{X=v_{k+1}\}>0$. Let $  X_i$, $i\ge 1 $  be  independent, identically distributed random variables having same law than $X$, and let $S_n=\sum_{j=1}^nX_j$, for each $n$. Let $\m_k\ge 0$ be such that $ \m= \sum_{k\in \Z}\m_k $ verifies $1-  \t_X<\m<1$, noting that $\t_X< 1$ always. Further let $\t=1-\m$, $s(t) =\sum_{k\in \Z} \m_k\, e^{ 2i \pi   v_kt}$ and    $\rho$ be such  that  $1-\t<\rho<1$.  
 We prove the following uniform semi-local theorems for the class $\mathcal F=\{F_{d},   d\ge 2\}$, where $F_{d}= d\N$. 
 \noi(i)
There exists  
$\theta=\theta(\rho,\t)$ with  $ 0< \theta <\t$,  
$C$  and $N$ 
 such that for $  n \ge N$,
\begin{align*}   \sup_{u\ge 0}\,\sup_{d\ge 2} \Big| \P \{ S_n+u\in F_{d}  \}   -  {1\over d}-  {1\over d}\sum_{ 0< |\ell|<d }& \Big( e^{ (i\pi {\ell\over   d }-{ \pi^2\ell^2\over 2 d^2}) } \t\,
 \E \,e^{2i \pi {\ell\over   d }\widetilde X  }   +s\big(  {\ell\over   d }\big)\Big)^n \Big|
  \cr  &\le    \frac{C }{ \theta ^{3/2}}\   \frac{(\log  n)^{5/2}}{ n^{3/2}}+2\rho^n.
\end{align*}
 \vskip 1 pt  \noi(ii) Let   $\mathcal D$ be a test set of divisors $\ge 2$,   $\mathcal D_\p$ be the section of $\mathcal D$ at height $\p$  and $|\mathcal D_\p|$ denote its   cardinality. Then,
    \begin{eqnarray*}
      \sum_{n=N}^\infty \   \sup_{u\ge 0} \, \sup_{\p\ge 2}\, {1\over |\mathcal D_\p |} \sum_{d\in \mathcal D_\p } \,\Big| \P \{d|S_n+u  \}   - {1\over d}\Big|
    & \le  &    \frac{C_1}{\t}
 \,  
         +  \frac{C_2 }{ \theta ^{3/2}} +\frac{2\rho^2}{1-\rho},
\end{eqnarray*}
where $C_1=\frac{2e^{     { \pi^2/ 4 }  }}
     {(1-e^{   -    {  \pi^2/ 16} }   ) }$, $C_2= C\,  \sum_{n=N}^\infty  \frac{(\log {n})^{5/2} }{ ({n})^{3/2}}$\,.      \end{abstract}

\section{Introduction}\label{s1}
   
\subsection{Problem studied}   Let $X  $  be  a square integrable random variable    taking values in a  lattice $\mathcal L(v_{
0},D )=v_0+D\,\Z$,
and let $\m={\mathbb E\,} X$, $\s={\rm Var }(
X)$.
Let $  X_i$, $i\ge 1 $  be  independent, identically distributed random variables having same law than $X$, and let $S_n=\sum_{j=1}^nX_j$, $n\ge1$. 
\vskip   2 pt 
The  probability 
 $\P\{S_n=N\}$ 
can be efficiently estimated by using Gnedenko's local limit theorem \cite{G}, which asserts that  \begin{equation}\label{llt.iid}   
 \sup_{
 N\in  \mathcal L( v_{ 0}n,D )}\Big|  \s \sqrt{n}\, {\mathbb P}\{S_n=N\}-{D\over  \sqrt{ 2\pi } }e^{-
{(N-n\m)^2\over  2 n\s^2} }\Big| = o(1),
\end{equation} 
if and only if the span $D$ is maximal,
(there are no  other real numbers
$v'_{0}
$ and
$D' >D$ for which
${\mathbb P}\{X
\in\mathcal L(v'_0,D')\}=1$).
\vskip 5 pt 
 Now consider the problem of estimating the probability 
$$ \P\{S_n\in F\}$$
where $F$ is a  possibly infinite  subset  of $\mathcal L( v_{ 0}n,D )$. This is obviously a quite important problem, which attracted a lot of attention. 
When $F$ moreover runs into a parametrized class $\mathcal F$ of subsets of $\mathcal L( v_{ 0}n,D )$, that problem can be   viewed as the search of  a semi-local   limit theorem with respect to   $\mathcal F$, thereby extending the classical   local   limit theorem to   parametrized classes of possibly infinite sets. 

\vskip 3 pt One can naturally think of using the local limit theorem for estimating $ \P\{S_n\in F\}$ or refinements of it. However, even on simple cases such as sets of multiples of an integer, the application of the local limit theorem  becomes rapidly unefficient.  
\vskip 5 pt 


Let 
\begin{equation}\label{multiples}\mathcal F=\{F_d, d\ge 2\} \qq \text{where} \qq F_d= d\,\N.
\end{equation}
\vskip 5 pt 
On the example below, we show   that    for Bernoulli sums already,  the application of a sharp form of the local limit theorem    fails from far to  provide a satisfactory bound. The question whether  
   a semi-local limit theorem for the class $\mathcal F$ defined above may be valid  for sequences  of i.i.d. square integrable $\mathcal L(v_{
0},1 ) $-valued random variables, has motivated us in this work and we will   prove that under moderate conditions it is indeed so, which is main result  of this paper. 
\vskip 5 pt 

\vskip 5 pt 

\subsection{An example} Let  $B_n=\b_1+\ldots+\b_n$,  where   $ \b_i  $ are i.i.d.  Bernoulli random variables.  On the one hand, 
 Theorem II in \cite{W1} provides a   uniform estimate on the whole range of  $d$, with   sharp rate of approximation.   
  \begin{theorem} \label{t1} 
   We have the following uniform estimate: 
\begin{equation}\label{unif.est.theta}\sup_{2\le d\le n}\Big|\P\big\{d|B_n\big\}- {\Theta(d,n)\over d}  \Big|= {\mathcal O}\big((\log n)^{5/2}n^{-3/2}\big),  
\end{equation}
where $\Theta (d,m) $ is the elliptic Theta function
\begin{equation}\label{theta}\displaystyle{ \Theta (d,m)  =  \sum_{\ell\in \Z} e^{im\pi{\ell\over   d }-{m\pi^2\ell^2\over 2 d^2}}. 
   }
   \end{equation}
 \end{theorem}

 The constant hidden in the symbol $\mathcal O$ is absolute. 
Special cases are also considered in \cite[Th.\,II]{W1}. 
It  is notably proved that  
\begin{equation}\label{(1.2)}
  \big|\P\big\{d|B_n\big\}- {1\over d}  \big|\le
    \begin{cases}  C\Big((\log n)^{5/2}n^{-3/2}+ { 1\over d}
 e^{ - {n \pi^2 \over 2d^2}} \Big)   & \quad  {\rm if}\  d\le \sqrt n,\cr 
      {C \over\sqrt n}   & \quad  {\rm if} \  \sqrt n\le d\le n.
\end{cases}
\end{equation}
 For any $\a>0$  
\begin{equation}\label{unif.est.theta.alpha} \sup_{d<  \pi   \sqrt{   n \over 2\a\log n}}\big|\P\big\{  d|   B_{ n} \big\}-{1\over d} 
\big|= {\mathcal O}_\e\big(n^{-\a+\e }\big),\qq \quad (\forall \e>0)  
\end{equation}   
 and for any $0<\rho<1 $, 
\begin{equation}\label{unif.est.theta.rho}  \sup_{d<  (\pi/\sqrt 2) n^{(1-\rho)/2} }\big|\P\big\{  d|   B_{ n} \big\}-{1\over d} 
\big|= {\mathcal O}_\e\big(e^{-(1-\e) n^\rho}\big),\qq \quad (\forall 0<\e<1). 
\end{equation}   
\vskip 3 pt On the other hand, the sharpest   version of the local limit theorem for Bernoulli sums is derived  from a fine local limit theorem with asymptotic expansion, see      \cite{P},  Ch. 7, Th.\,13., and states as follows,       
   \begin{eqnarray}\label{lltber} \sup_{z}\, \Big|  \P\big\{B_n=z\} 
 -\sqrt{\frac{2}{\pi n}} e^{-{ (2z-n)^2\over
2 n}}\Big| =o\Big(\frac{1}{n^{3/2}}\Big). 
  \end{eqnarray} 
    Although the   error term   is   sharp, \eqref{lltber}  will yield 
   a less precise result than Theorem \ref{t1}. 
   \vskip 2 pt
Indeed, at first there is a numerical constant $a$ such that
\begin{eqnarray*}  \P\{  \big|{2B_n-n\over \sqrt n } \big|> a\sqrt{     \log n}\}    \le    n^{-1} , \qq \quad n\ge 1 .
\end{eqnarray*}
 Thus using \eqref{lltber},
 \begin{eqnarray*}
 \P \{ d | B_n \}&=&  \sum_{|z-n/2|\le a\sqrt{    n \log n}\atop z\equiv 0\, (d)} \P\big\{B_n=z\}+ \mathcal
O(n^{-1}) 
\cr &=&    \sqrt{\frac{2}{\pi n}}\sum_{|z-n/2|\le 
 a\sqrt{    n \log n}\atop z\equiv 0\, (d)} e^{-{ (2z-n)^2\over
2 n}}+ o(  \frac{\sqrt{  \log n}}{n}   ).
\end{eqnarray*}

Besides, noticing that $2z-n$ is integer, we have for $n\ge 2$  
\begin{eqnarray*}
 \sum_{|z-n/2|> a\sqrt{    n \log n} } e^{-{ (2z-n)^2\over
2 n}}
&\le & \sum_{ Z> 2a\sqrt{    n \log n} \atop Z\in \N} e^{-{ Z^2\over
2 n}}\ \le \ \int_{a\sqrt{    n \log n}}^\infty e^{-{t^2\over 2n}} \dd t 
\cr &=&  \sqrt n\int_{a\sqrt{      \log n}}^\infty e^{-{u^2\over 2n}} \dd u\ = \ o(n^{-1}).
\end{eqnarray*}
 Consequently \eqref{lltber} implies,
\begin{eqnarray} \label{dllt0}\sup_{2\le d\le n}\Big|\P \{ d | B_n \}-\sqrt{\frac{2}{\pi n}}\sum_{ 
z\equiv 0\, (d)} e^{-{ (2z-n)^2\over 2 n}}\Big|= o(  \frac{\sqrt{  \log n}}{n}  ).
\end{eqnarray} 
\vskip 5 pt 

We now compare \eqref{unif.est.theta} with \eqref{dllt0}  by  using  Poisson summation formula which we recall: for   $x\in \R,\  0\le
\d\le 1
$, 
\begin{equation}\label{poisson}\sum_{\ell\in \Z} e^{-(\ell+\d)^2\pi x^{-1}}=x^{1/2} \sum_{\ell\in \Z}  e^{2i\pi \ell\d -\ell^2\pi x}.
 \end{equation}  
Applying it  with  $x=\pi
n/(2d^2)$, 
$\d=\{n/(2d)\}$, gives
\begin{equation}\label{poisson.comp.theta.llt}\frac{\Theta(d,n)}{ d}  =\sqrt{  { 2\over \pi
n}}  \sum_{ 
z\equiv 0\, (d)} e^{-{ (2z-n)^2\over 2 n}}.
 \end{equation} 
   Estimate  \eqref{unif.est.theta} thus  implies
\begin{equation}\label{poisson.comp.theta.unif}   \sup_{2\le d\le n}\Big|\P\big\{d|B_n\big\}- \sqrt{  { 2\over \pi
n}}\sum_{ 
z\equiv 0\, (d)} e^{-{ (2z-n)^2\over 2 n}}  \Big|=
{\mathcal O}\Big({ \log^{5/2} n  \over n^{ 3/2}}\Big),
\end{equation} 
 which is clearly  much better   than   (\ref{dllt0}).  




\section{Main result}\label{s2}
In this paper we  extend 
the semi-local limit 
Theorem \ref{t1} to sums of i.i.d. square integrable random variables. Clearly this   cannot be derived  from Gnedenko's local limit theorem. Note by the way that according to Matskyavichyus   \cite{Mat}, Gnedenko's theorem is optimal.  In place we shall in a first step prove a drifted version of  Theorem \ref{t1} for value distribution of divisors of $B_n + u$, which turns up to be uniform in $u$, and next combine it with a coupling method called the  \lq\lq {\it Bernoulli part extraction}\rq\rq\, of a random variable. This method  is essentially  due to Mc Donald \cite{M}, \cite{MD}, see also \cite{GW3}, and was successfully applied to the local limit theorem.  However, the problem investigated here being more complicated than the usual local limit theorem, it was necessary to  refine it somehow.
 The approach we propose is flexible and should be adaptable  to other parametrized classes of sets.

\vskip 5pt
 Let $X$ be
  a random variable   such that  $\P\{X
\in\mathcal L(v_0,1)\}=1$, and let $f(k)=\P\{X=v_k\}$ for all $k\in \Z$.   
 We use the following structural characteristic \begin{eqnarray}\label{varthetaX.def}  \t_X  :=\sum_{k\in \Z}f(k)\wedge f(k+1), 
\end{eqnarray} 
and note that $0\le \t_X<1$ always. We assume that  the following condition   is satisfied,
\begin{eqnarray}\label{varthetaX}  \t_X >0.
\end{eqnarray} 

This is possible only if the maximal span of $X$ is $1$, which we do assume. In our approach  we associate   to   $X$,  another random variable $\widetilde X$ close to $X$,  which we shall now describe. 

\vskip 3 pt First let $ \bar \m= \{ \m_k, k\in \Z\}$ be a given sequence of non-negative reals satisfying the following simple condition: 
\begin{eqnarray}\label{muk}\qq\hbox{\it{ For all $k\in \Z$, $0<\m_k < f(k) $ if $f(k)>0$, and $\m_k=0$ if $f(k)=0$}. }
\end{eqnarray}

\vskip 3 pt \noi Let  $\m= \sum_{k\in \Z}\m_k$, and assume that 
\begin{eqnarray}\label{mu}
1-\m<\t_X. 
\end{eqnarray} 

\vskip 3 pt Next let  $\bar \tau=  \{ \tau_k, k\in \Z\}$ be a sequence of non-negative reals satisfying  the equation   
\begin{equation}\label{tau.mu} {\tau_{k-1}+\tau_k\over 2} =f(k)  - \m_k,
\end{equation}  for all $k \in \Z$.

 \vskip 2 pt 
 This equation is solvable, the solutions are given in \eqref{tau.mu1.sol}. By construction 
 $ \tau_{k-1}+\tau_k\le 2f(k)$, for all $k \in \Z$. 
\vskip 3 pt 
 
 Put 
\begin{equation}\label{basber1}   \t = \sum_{k\in \Z}  \tau_k = 1- \m.
\end{equation}
The sequence  $\bar \tau$ only depends on the  random variable $X$ and the sequence $\bar \m$.   We associate  to $X$  and $\bar \m$ a random variable $\widetilde X$  defined by the relation
\begin{equation}\label{widetildeX} \P\{ \widetilde X =v_k\}=\frac{\tau_k}{\t}, \qq  k\in \Z   .
\end{equation}
    
 Note that     as   $\tau_k$, $\m_k$ are defined independently of $n$, and so $\widetilde X$ is. 
  
\vskip 3 pt Our main result states as follows.

\begin{theorem} \label{t2} Let $X  $  be  a square integrable random variable    taking values in a  lattice $\mathcal L(v_{
0},1 )$, with maximal span $1$ and satisfying condition \eqref{varthetaX}. Let $  X_i$, $i\ge 1 $  be  independent, identically distributed random variables having same law than $X$, and let $S_n=\sum_{j=1}^nX_j$, for each $n$. Further let $s(t) =\sum_{k\in \Z} \m_k\, e^{ 2i \pi   v_kt}$,      $\rho$ be such  that  $1-\t<\rho<1$, $\t$ being defined in \eqref{basber1}, and $\widetilde X$ be defined in \eqref{widetildeX}.
  \vskip 2 pt  {\rm (i)}  There exists
$\theta=\theta(\rho,\t)$ with  $ 0< \theta <\t$ and  
$C$ and $N$ 
 such that we have for all $  n \ge N$,
\begin{align*}    \sup_{u\ge 0}\,\sup_{d\ge 2} \Big| \P \{d|S_n +u \}   -  {1\over d}-  {1\over d}\sum_{ 0< |\ell|<d } \Big( e^{ (i\pi {\ell\over   d }-{ \pi^2\ell^2\over 2 d^2}) } &\t\,
 \E \,e^{2i \pi {\ell\over   d }\widetilde X  }   +s\big(  {\ell\over   d }\big)\Big)^n \Big|
  \cr  &\le    \frac{C }{ \theta ^{3/2}}\   \frac{(\log  n)^{5/2}}{ n^{3/2}}+2\rho^n.
\end{align*}

{\rm (ii)} Let   $\mathcal D$ be a test set of divisors $\ge 2$,   $\mathcal D_\p$ be the section of $\mathcal D$ at height $\p$  and $|\mathcal D_\p|$ denote its   cardinality. Then,
    \begin{eqnarray*}
     \sum_{n=N}^\infty \,  \sup_{u\ge 0} \, \sup_{\p\ge 2}\, {1\over |\mathcal D_\p |} \sum_{d\in \mathcal D_\p } \,\Big| \P \{d|S_n+u  \}   - {1\over d}\Big|
   & \le  &    \frac{C_1}{\t}
 \,  
         +  \frac{C_2 }{ \theta ^{3/2}} +\frac{2\rho^2}{1-\rho},
\end{eqnarray*}
where $C_1=\frac{2e^{     { \pi^2/ 4 }  }}
     {(1-e^{   -    {  \pi^2/ 16} }   ) }$, $C_2= C\,  \sum_{n=N}^\infty  \frac{(\log {n})^{5/2} }{ ({n})^{3/2}}$.\end{theorem}
The proof 
 is very delicate, the second assertion is proved by introducing two tricks, new in this context.\begin{remark} 
The used 
 characteristic $\t_X$ is tigthly related to the \lq\lq smoothness\rq\rq characteristic
\begin{eqnarray}\label{delta}  \d_X =\sum_{m\in \Z}\big|{\mathbb P}\{X=m\}-{\mathbb P}\{X=m-1\}\big|, 
\end{eqnarray}
since $
\d_X =2(1- \t_X)  $, as quoted in  Mukhin
 \cite{Mu}, p.\,700, noticing that $a+b-|a-b|= 2(a\wedge b)$, $a$ and $b$ non-negative. That  characteristic was  introduced and much investigated by Gamkrelidze in several original works, \cite{Gam1},    \cite{G2}, \cite{G1} notably, also in remarkable counter-examples.
     We thus have the equivalence 
 \begin{equation}\label{equiv}  \t_X>0\quad \Longleftrightarrow \quad \d_X<2.
\end{equation}  
An important  consequence is
 that  condition $\d_X<2$ implies that $X$ has a Bernoulli component. 
 This follows from  Lemma \ref{bpr}. The  two characteristics seem through formula \eqref{equiv}  equivalent. However 
 the second is used in Gamkrelidze's works in relation with the method of characteristic  functions, whereas the first is related to the Bernoulli part extraction, which is known to be characteristic function   free.
  
  \end{remark}

\begin{remark} \label{llt.transf.} For a random variable $X  $      taking values in a  lattice $\mathcal L(v_{
0},D)$,     the linear transformation
\begin{equation*} 
 X'_j= \frac{X_j-v_0}{D},
  \end{equation*}
allows one to reduce  to the case  $v_0=0$, $D=1$.
\end{remark}
\vskip 10 pt The paper is organized as follows. In the next Section we prove  a uniform version of Theorem \ref{t1} with drift. In Section \ref{s4} we collect the necessary auxiliary results.  Finally in Section \ref{s5} we give the proof of Theorem \ref{t2}.

\vskip 8 pt \noi {\it Notation.} Throughout    $C$ denotes an absolute constant whose value may change at each occurence.


\section{An intermediate result} \label{s3}


For the proof of the main result, an  extension of   Theorem \ref{t1} to value distribution of divisors of $B_{ n}+u$,   $u $ being any non-negative integer, is necessary. Interestingly enough, the  estimate we prove is   uniform over all $u\ge 0 $. We begin with a preliminary  observation. The restriction $d\le n$ in the estimate of Theorem \ref{t1} is superfluous. That  estimate is in turn also uniform over all $d\ge 2$.  
 If $d>n$, as the first term is 0,   \eqref{unif.est.theta} provides an   estimate  of   the central term ${\Theta(d,n)/ d} $, for $d>n$. 
 
This point   has some  degree of importance as it will be used in the course of the proof of  the main result.   The proof of the Theorem below  is   similar to the one of Theorem \ref{t1}, which corresponds to the case $u=0$,  see proof of  Theorem II in \cite{W1}. We will provide the necessary details to understand this point, notably. 
 \begin{theorem}\label{estPd.Bn.u} 
There exists two absolute constants $C$ and $n_0$ such that for all $n\ge n_0$,
\begin{equation*}
 \sup_{u\ge 0}\,\sup_{d\ge 2}\Big|\P\big\{  d|   B_{ n}+u \big\}-  {1\over d}\sum_{ 0\le |j|< d }
e^{i\pi (2u+n){j\over d}}\  e^{  -n  
{\pi^2j^2\over 2d^2}}\Big|\le C\, (\log n)^{5/2}n^{-3/2}.
\end{equation*}

\end{theorem}

Put 
 \begin{equation}\label{theta.u.}
\Theta_u(d,n)  =  \sum_{\ell\in \Z}  e^{i\pi (2u+n){j\over d}}\  e^{  -n  
{\pi^2j^2\over 2d^2}}.
 \end{equation}
 Note that $\Theta_0(d,n)=\Theta(d,n)$.
As a corollary we get,
\begin{corollary}\label{cor.estPd.Bn.u} 
For some absolute constant $C$, we have  
\begin{equation*} 
\sup_{u\ge 0}\, \sup_{d\ge 2}\Big|\P\big\{  d|   B_{ n}+u \big\}-{\Theta_u(d,n)\over d} \Big| \le C \,(\log
n)^{5/2}n^{-3/2} ,
\end{equation*}
for all $n\ge n_0$.
\end{corollary}

\begin{proof}[Proof of Theorem \ref{estPd.Bn.u}]    Let  $d,n,u$ be arbitrary non-negative integers. As
\begin{equation}
d\d_{d |B_{n }+u}=\sum_{j=0}^{d-1} e^{2i\pi  {j \over d}(B_{n }+u)},
\end{equation}
 we obtain  after integration
 \begin{eqnarray}\label{basic}
 \P\big\{ d | B_n+u\big\}&=&{1\over d}\sum_{j=0}^{d-1} e^{2i\pi  {j \over d}u}\E e^{2i\pi  {j \over d} B_{n }}\,=\, {1\over d}\sum_{j=0}^{d-1} e^{2i\pi  {j \over d}u}  \Big(\frac{e^{2i\pi   {j \over d} }+1}{2}\Big)^n
 \cr&=&{1\over d}\sum_{j=0}^{d-1} e^{i\pi (2u+n){j\over d}}\big(\cos { \pi j\over d}\big)^{n}
 .
\end{eqnarray}
 
 We first operate a reduction due to symmetries. The following   Lemma is formula (2.3) in  \cite{W1}, which is only stated. As we shall see, the role of these symmetry properties is important and we have  included a detailed proof.

 \begin{lemma}\label{reduction.sym} For any integers $d\ge 2$, $n\ge 2$ and $u\ge 0$, 
 \begin{eqnarray*}
  \P\big\{ d| B_n+u\big\}&=& {1\over d} + {2\over d}\sum_{1\le j< d/2} \cos\big(\pi (2u+n){j\over d} \big)\big(\cos { \pi j\over d}\big)^{n}
.
\end{eqnarray*}
 \end{lemma}
 
 \begin{proof} Let $d=2\d$  with $\d \ge 1$. Then 
 \begin{equation}\label{basica}
 \P\big\{ 2\d | B_n+u\big\} \,=\, {1\over 2\d}\sum_{j=0}^{\d-1} e^{i\pi (2u+n){j\over 2\d}}\big(\cos { \pi j\over 2\d}\big)^{n}
  +{1\over 2\d}\sum_{j=\d}^{2\d-1} e^{i\pi (2u+n){j\over 2\d}}\big(\cos { \pi j\over 2\d}\big)^{n}
\end{equation}
 
Letting first $j=2\d-v$, $v=1,\ldots, \d$ in the last sum, we have  
 \begin{eqnarray*}\sum_{j=\d}^{2\d-1} e^{i\pi (2u+n){j\over 2\d}}\big(\cos { \pi j\over 2\d}\big)^{n}
&=&\sum_{v=1}^\d  e^{i\pi (2u+n){2\d-v\over 2\d}}\big(\cos { \pi(2\d-v)\over 2\d}\big)^{n}
\cr &=&\sum_{v=1}^\d  e^{i  (2u+n)(\pi -{\pi v\over 2\d})}\big(\cos (\pi -{\pi v\over 2\d})\big)^{n}. 
\end{eqnarray*}
But 
\begin{equation}\label{formula.sym}e^{i  m(\pi-x)} \cos^n    (\pi-x)  =(-1)^me^{-i 
mx}  (-1)^n\cos^n   x= (-1)^{m+n}e^{-i 
nx}  \cos^n   x  .
\end{equation} Thus
\begin{eqnarray*}\sum_{j=\d}^{2\d-1} e^{i\pi (2u+n){j\over 2\d}}\big(\cos { \pi j\over 2\d}\big)^{n}
  &=&\sum_{v=1}^\d  e^{i  (2u+n)(\pi -{\pi v\over 2\d})}\big(\cos (\pi -{\pi v\over 2\d})\big)^{n}
  \cr &=&
  \sum_{v=1}^\d  e^{-i  (2u+n) {\pi v\over 2\d}}\big(\cos {\pi v\over 2\d}\big)^{n}
  \cr &=&
  \sum_{v=1}^{\d -1} e^{-i  (2u+n) {\pi v\over 2\d}}\big(\cos {\pi v\over 2\d}\big)^{n}, 
\end{eqnarray*}
since for $v=\d$, we have $e^{-i  (2u+n) {\pi v\over 2\d}}\big(\cos {\pi v\over 2\d}\big)^{n}=e^{-i  (2u+n) {\pi  \over 2 }}\big(\cos {\pi  \over 2 }\big)^{n}=0$.
\vskip 3 pt 
Carrying   this back to \eqref{basica} gives,
 \begin{eqnarray}\label{basicb}
& & \P\big\{ 2\d | B_n+u\big\} \,=\, {1\over 2\d}\sum_{j=0}^{\d-1} e^{i\pi (2u+n){j\over 2\d}}\big(\cos { \pi j\over 2\d}\big)^{n}
 \cr & &\quad+{1\over 2\d} \sum_{j=1}^{\d-1}  e^{-i  (2u+n) {\pi j\over 2\d}}\big(\cos {\pi j\over 2\d}\big)^{n}
 \cr &=&{1\over 2\d}+ {1\over 2\d} \sum_{j=1}^{\d-1} \Big\{e^{i\pi (2u+n){j\over 2\d}}+e^{-i  (2u+n) {\pi j\over 2\d}}\Big\}\big(\cos {\pi j\over 2\d}\big)^{n}
 \cr &=&{1\over 2\d}+{2\over 2\d}\sum_{j=1}^{\d-1} \cos \big(\pi (2u+n){j\over 2\d}\big)\big(\cos { \pi j\over 2\d}\big)^{n}.
\end{eqnarray}

Now let  $d=2\d+1$ with $\d \ge 1$ and write that
 \begin{eqnarray}\label{basicc}
 \P\big\{ 2\d+1 | B_n+u\big\}&=& {1\over 2\d+1}\sum_{j=0}^{\d} e^{i\pi (2u+n){j\over 2\d +1}}\big(\cos { \pi j\over 2\d +1}\big)^{n}
\cr & &\quad +{1\over 2\d+1}\sum_{j=\d+1}^{2\d} e^{i\pi (2u+n){j\over 2\d +1}}\big(\cos { \pi j\over 2\d +1}\big)^{n}.
\end{eqnarray} 

With the variable change $j=2\d +1-v$,  $v=1, \ldots \d$,    the second sum writes using \eqref{formula.sym},
\begin{eqnarray}\label{basicd}
  \sum_{j=\d}^{2\d} e^{i\pi (2u+n){j\over 2\d +1}}\Big(\cos { \pi j\over 2\d}\Big)^{n}
&=&  \sum_{v=1}^{\d} e^{i\pi (2u+n){(2\d +1-v)\over 2\d +1}}\Big(\cos { \pi (2\d +1-v)\over 2\d+1}\Big)^{n}
\cr &=&    \sum_{v=1}^{\d} e^{-i\pi (2u+n){v\over 2\d +1}}\big(\cos { \pi v\over 2\d+1}\big)^{n}.
\end{eqnarray} 
Thus 
\begin{eqnarray}\label{basicc1}
& & \P\big\{ 2\d +1| B_n+u\big\}\,=\, {1\over 2\d+1}\sum_{j=0}^{\d} e^{i\pi (2u+n){j\over 2\d +1}}\big(\cos { \pi j\over 2\d +1}\big)^{n}
\cr & &\quad +{1\over 2\d+1}\sum_{j=1}^{\d} e^{-i\pi (2u+n){j\over 2\d +1}}\big(\cos { \pi j\over 2\d+1}\big)^{n}
\cr &=& {1\over 2\d+1} + {1\over 2\d+1}\sum_{j=1}^{\d} \Big\{ e^{i\pi (2u+n){j\over 2\d +1}}+e^{-i\pi (2u+n){j\over 2\d +1}}\Big\}\big(\cos { \pi j\over 2\d +1}\big)^{n}
\cr &=& {1\over 2\d+1} + {2\over 2\d+1}\sum_{j=1}^{\d} \cos\big(\pi (2u+n){j\over 2\d +1} \big)\big(\cos { \pi j\over 2\d +1}\big)^{n}
.
\end{eqnarray} 
From \eqref{basicb} and \eqref{basicc1} follows that
\begin{eqnarray}\label{basicd1}
  \P\big\{ d| B_n+u\big\}&=& {1\over d} + {2\over d}\sum_{1\le j< d/2} \cos\big(\pi (2u+n){j\over d} \big)\big(\cos { \pi j\over d}\big)^{n}
.
\end{eqnarray} 
 \end{proof}
Thanks to the reduction operated, we can work in the first quadrant, instead of the half-circle, which will   permit us  later,   to get  in some particular remarkable     cases of divisors of $n$ close to $\sqrt n$,  strong  improvements of the general estimate we are now going to prove. 

\vskip 3 pt Let 
 $\a>\a'>0$. Let
 \begin{equation}\label{phi.tau}\p_n=   \big( {2\a\log n \over n}\big)^{1/2},
\qquad\qquad \tau_n= {\sin\p_n/2\over
\p_n /2}.
\end{equation} 
We assume
$n$ sufficiently large,   $n\ge
n(\a, \a')$ say, so that 
\begin{equation} \label{ntaun}
  \tau_n\ge (\a^\prime/\a)^{1/2},\qq\quad (\forall n\ge n(\a, \a')).
\end{equation}  

Consider two sectors
  $$A_n = ]0,\p_n[,  \qq\qq A'_n=[ \p_n, {\pi\over 2} [.$$ 
\noindent   If ${\pi j\over d}\in A'_n$, then $|\cos {\pi j\over d}|\le \cos \p_n$. And 
 $|\cos {\pi j\over d}|^n \le \big(\cos \p_n\big)^n  \le e^{-2n \sin^2(\p_n/2)} $. 
As $ 2n \sin^2(\p_n/2) = 2n (\p_n/2)^2\tau_n^2\ge \a^\prime\log n$,  
  we deduce 
\begin{equation}\label{(2.5)} \sum_{ 1\le j  <d/2\ :\ { \pi j\over d}\in A'_n}\big|\cos { \pi j\over d}\big|^{n }\le
d\,n^{-\a^\prime }/2.
\end{equation}
 \vskip 5 pt
 We also have  if ${\pi j\over d}\in A'_n$,  that ${\pi \over 2}>{\pi j\over d}\ge \p_n=  \big( {2\a\log n
\over n}\big)^{1/2}
$. Thus  
\begin{equation}\label{(2.10)} \sum_{ 1\le j  <d/2:\, { \pi j\over d}\in A'_n} e^{  -n   {\pi^2j^2\over 2d^2}}\le
dn^{- \a  }/2,
\end{equation}
which  complements \eqref{(2.5)}.
\vskip 5 pt Now assume  $\a>\a'>3/2$.   
Consider the contribution of the terms for which ${\pi j\over d}\in A_n$. Let 
$$D= \sum_{ 1\le j  <d/2 \, :\, { \pi j\over d}\in A_n}\cos \big(\pi (2u+n){j\over d}\big) \Big(\cos^{n } { \pi j\over d}  - e^{ -n   {\pi^2j^2\over
2d^2}} \Big). $$ 
By using the elementary inequality:  
$|e^u-e^v|\le |u-v| $ for $u,v\le 0$ we get 
$$|D|\le n \sum_{ 1\le j  <d/2 \, :\, { \pi j\over d}\in  A_n} \big|\log \cos  { \pi j\over d}  +      {\pi^2j^2\over 2d^2} \big|. $$
Since
$\log(1-2\sin^2(x/2))= -x^2/2 +{\mathcal O}(x^4)$ near
$0$  and $A_n=]0,\p_n[$, we deduce  
\begin{eqnarray}\label{(2.11)} |D|\,\le\, n \sum_{ 1\le j  <d/2 \, :\, { \pi j\over d}\in  ]0,\p_n[} \big|\log \cos  { \pi j\over d}  +      {\pi^2j^2\over 2d^2} \big|&\le &Cn\sum_{ 1\le j  <d/2\, :\, { \pi j\over d}
\in ]0,\p_n[}  ({j\over
d})^4 
\cr &\le &    {Cn\over
d^4}  \sum_{     j   \le  {d\over \pi} ({2\a\log n \over n} )^{1/2}}   j^4
 \cr & \le & C_\a d(\log n)^{5/2}n^{-3/2} .
 \end{eqnarray}

Combining \eqref{(2.5)}, \eqref{(2.10)} and \eqref{(2.11)} shows that
\begin{eqnarray}
& &\Big|\sum_{ 1\le j  <d /2 }\cos \big(\pi (2u+n){j\over d}\big)\ \Big\{ \cos^{n } { \pi j\over d} - e^{ -n   {\pi^2j^2\over
2d^2}}\Big\} \Big|
\cr &\le& \Big|\sum_{ 1\le j  <d /2 \, :\, { \pi j\over d}\in A'_n} \cos \big(\pi (2u+n){j\over d}\big)\ \Big\{ \cos^{n } { \pi j\over d}  -e^{ -n   {\pi^2j^2\over
2d^2}}\Big\} \Big|
\cr & &\qq + \Big|\sum_{ 1\le j  <d /2 \, :\, { \pi j\over d}\in A_n} \cos \big(\pi (2u+n){j\over d}\big)\ \Big\{ \cos^{n } { \pi j\over d} - e^{ -n   {\pi^2j^2\over
2d^2}}\Big\} \Big|
\cr &\le&  dn^{-\a^\prime }/2+ dn^{- \a  }/2+C_\a d(\log n)^{5/2}n^{-3/2}
\cr   &\le  &  C_\a d(\log
n)^{5/2}n^{-3/2}.
\cr   &   & \end{eqnarray}
Dividing both sides by $d$, and reporting next the obtained estimate into \eqref{basic}   gives in view of Lemma \ref{reduction.sym}
\begin{equation}\label{(2.13)} \Big|\P\big\{ d | B_n+u\big\}-{1\over d}- {1\over d}\sum_{ 1\le |j|< d/2 }
e^{i\pi (2u+n){j\over d}}\  e^{  -n  
{\pi^2j^2\over 2d^2}}\Big|\le C_\a  (\log
n)^{5/2}n^{-3/2}.
\end{equation}

Noticing now that   no condition on $d\ge 2$ is made, and that  all constants involved in the above calculations are independent from $u$, 
we obtain  by giving a value $>3/2$ to $\a$, $\a'$ and noting $n_0= n(\a,\a')$, $C=C_\a$,   that for all $n\ge n_0$,
\begin{equation}\label{(2.14)}
\sup_{u\ge 0}\,  \sup_{d\ge 2}\Big|\P\big\{  d|   B_{ n}+u \big\}-  {1\over d}\sum_{ 0\le |j|< d }
e^{i\pi (2u+n){j\over d}}\  e^{  -n  
{\pi^2j^2\over 2d^2}}\Big|\le \, C\, (\log n)^{5/2}n^{-3/2}.
\end{equation}
This  establishes Theorem \ref{estPd.Bn.u}.
 \end{proof}

\begin{remark}The proof can be summarized as follows, if 
\begin{equation}
A=  \sum_{ 1\le j  <d/2\atop{ \pi j\over d}\in A_n}  \big|\cos^{n } { \pi j\over d}  - e^{ -n   {\pi^2j^2\over
2d^2}}  \big|, \qq A'= \sum_{ 1\le j  <d/2\atop{ \pi j\over d}\in A'_n}  \big(|\cos  { \pi j\over d}|^n + e^{ -n   {\pi^2j^2\over
2d^2}}  \big) 
 \end{equation}
 then for all $n\ge n_0$ and $d\ge 2$, $A+A'\le    C  \,d\,(\log n)^{5/2}n^{-3/2}$.  
  The  {\it cosine} parts represent $\P\{d|B_n +u\}$, and the {\it exponential } parts the corresponding partial sums of $\Theta_u(d,n)$. 
 \end{remark}

 \begin{remark}
The proof given  is transposable to 
  other systems of independent random variables when such symmetries exist. This is not the case for  the   Hwang-Tsai model of the Dickman function, neither for  the Cram\'er model of primes for instance, see \cite{W2}. 
\end{remark}
\begin{proof}[Proof of Corollary \ref{cor.estPd.Bn.u}]
Recall that by \eqref{theta.u.},
$$\Theta_u(d,n)  =  \sum_{\ell\in \Z}  e^{i\pi (2u+n){j\over d}}\  e^{  -n  
{\pi^2j^2\over 2d^2}}.
$$
Consider the remainder $r:=\sum_{j\ge d/2} e^{-n{\pi^2j^2\over 2d^2}}$. By Theorem \ref{estPd.Bn.u} and using the
triangle inequality,  
\begin{equation}\label{(2.15)}\Big|\P\big\{  d|   B_{ n}+u \big\}-{1\over d}\sum_{ \ell\in\Z}
e^{i\pi (2u+n){\ell\over d}}\  e^{  -n  
{\pi^2\ell^2\over 2d^2}}\Big| \le C \,(\log
n)^{5/2}n^{-3/2}+ {2r\over d}.
\end{equation}
We prove that for all integers $d\ge 2$ and $n\ge 2$,
\begin{eqnarray}\label{estimate.r} r \,\le\,   C\, e^{- {\pi^2n \over 72}}.
\end{eqnarray} 
\noi --- If $d=2$, 
\begin{eqnarray*}r\ =\  \sum_{j=1}^\infty e^{-n{\pi^2j^2\over 8}}
&\le  & e^{- {\pi^2n \over 8}}+ \sum_{j=2}^\infty \int_{{\pi (j-1)
 \over  2\sqrt 2}}^{{\pi j \over 2\sqrt 2}}e^{-   n x^2 }dx\ = \ e^{- {\pi^2n \over 8}}+   \int_{{\pi 
 \over  2\sqrt 2}}^\infty e^{-   n x^2 }dx
 \cr (x={y\over  \sqrt{ 2n}})\quad &  =&  e^{- {\pi^2n \over 8}}+  
 \int_{{\pi\sqrt n  
 \over  2 }}^\infty e^{-   y^2/2 }{dy\over\sqrt{ 2n}}\ \le \ Ce^{- {\pi^2n \over 8}}.
 \end{eqnarray*}
 
\smallskip\par
\noi --- If $d\ge 3$, then 
$${{d\over 2}-1\over d}\ge {{d\over 2}-{d\over 3}\over d}={1\over 6}. $$
Therefore
\begin{eqnarray}\label{estimate.r.} r  \ \le \  \sum_{j\ge d/2}^\infty \int_{{\pi (j-1)
\over   \sqrt 2 d}}^{{\pi j \over  \sqrt 2d}}e^{-   n x^2 }dx&\le &  \int_{{\pi({d\over 2}-1) 
\over   \sqrt 2 d}}^\infty e^{-   n x^2 }dx
\cr(x={y\over  \sqrt{ 2n}})\quad  & \le  &  \int_{{\pi  
\over  6 \sqrt 2 }}^\infty e^{-   n x^2 }dx    \int_{{\pi \sqrt n 
\over  6   }}^\infty  e^{-   y^2/2 }{dy\over\sqrt{ 2n}}\ \le \ Ce^{- {\pi^2n \over 72}},
\end{eqnarray} 
a bound which is thus valid for all integers $d\ge 2$ and $n\ge 2$.  
\vskip 2 pt 
Incorporating now    these
estimates into \eqref{(2.15)}, gives
\begin{equation}\label{(2.15a)}\Big|\P\big\{  d|   B_{ n}+u \big\}-{\Theta_u(d,n)\over d}\, \Big| \le C \,(\log
n)^{5/2}n^{-3/2} ,
\end{equation}
for $n\ge n_0$. The latter estimate being uniform in $u\ge 0$ and $d\ge 2$, in view of Theorem \ref{estPd.Bn.u} and \eqref{estimate.r}, this achieves the proof. 
\end{proof}

\vskip 3pt

\begin{corollary}[Special cases]\label{special.cases}{\rm (i)} For each $\a\!>\!\a'\!>\!0$ and $n$ such that $  \tau_n\ge (\a^\prime/\a)^{1/2}$, where $\tau_n$ is defined in \eqref{phi.tau}, we have 
\begin{equation*}\label{ntaun.est.alpha} \sup_{u\ge 0}\,\sup_{d<  \pi   \sqrt{   n \over 2\a\log n}}\Big|\P\big\{  d|   B_{ n} +u\big\}-{1\over d} 
\Big|\,\le\, n^{-\a'}.
\end{equation*}
{\rm (ii)}
Let $0<\rho<1 $.  Let also $0<\e<1$, and suppose $n$ sufficiently large so that $\widetilde\tau_n\ge \sqrt{1-\e}$, where
$$  \widetilde\tau_n= {\sin\psi_n/2\over
\psi_n /2}\qq \qq \psi_n= \big({2n^\rho \over n}\big)^{1/2}.$$ 
Then,
\begin{equation*}  \sup_{u\ge 0}\,\sup_{d<  (\pi/\sqrt 2) n^{(1-\rho)/2} }\Big|\P\big\{  d|   B_{ n} +u\big\}-{1\over d} 
\Big|\,\le\, e^{-(1-\e) n^\rho}.
\end{equation*}
  \end{corollary}
\begin{proof} (i) 
If $d<  \pi   \sqrt{   n \over 2\a\log n} $, then
$${ \pi j\over d}> { \pi  \over \pi   \sqrt{   n \over 2\a\log n}}= \sqrt{   2\a\log n \over   n}=\p_n,$$
and so  $\{ 1\le j  <d/2:{ \pi j\over d}\in A_n\}=\emptyset$. 
 \vskip 3 pt In view of Lemma \ref{reduction.sym}, \eqref{(2.5)}, we get: 
  {\it For each $\a\!>\!\a'\!>\!0$ and $n$ such that $  \tau_n\ge (\a^\prime/\a)^{1/2}$,
   we have} 
\begin{equation}\label{ntaun.est.alpha.} \sup_{u\ge 0}\,\sup_{d<  \pi   \sqrt{   n \over 2\a\log n}}\Big|\P\big\{  d|   B_{ n}+u \big\}-{1\over d} 
\Big|\,\le\, n^{-\a'}.
\end{equation}

(ii) Now, let $0<\rho<1 $. 
Consider the modified sectors
  $$\widetilde A_n = ]0,\psi_n[,  \qq\qq \widetilde A'_n=[ \psi_n, {\pi\over 2} [.$$ 
where  
$$\psi_n= \big({2n^\rho \over n}\big)^{1/2}  \qquad\qquad \widetilde\tau_n= {\sin\psi_n/2\over
\psi_n /2}.$$

Let also $0<\e<1$, and suppose $n$ sufficiently large for $\widetilde\tau_n$ to be greater than $\sqrt{1-\e}$.    Exactly as before, if
$ {\pi j\over d}\in
\widetilde A'_n$, then
$|\cos {\pi j\over d}|\le
\cos
\psi_n$, so that
 $|\cos {\pi j\over d}|^n \le  (\cos \psi_n )^n  \le e^{-2n \sin^2(\psi_n/2)} $. 
And $ 2n \sin^2(\psi_n/2) = 2n (\psi_n/2)^2\tau_n^2=     n^\rho    \tau_n^2\ge (1-\e) n^\rho $. We deduce  
 $$ \sum_{ 1\le j  <d/2\ :\ { \pi j\over d}\in \widetilde A'_n}\big|\cos { \pi j\over d}\big|^{n }\le
de^{-(1-\e) n^\rho}/2.$$
Since  $\psi_n=\big({2n^\rho \over n}\big)^{1/2}
\le {\pi j\over d}<{\pi \over 2}$,  we further get
\begin{equation}\label{(2.10).rho} \sum_{ 1\le j  <d/2\ :\ { \pi j\over d}\in A'_n} e^{  -n   {\pi^2j^2\over 2d^2}}\le
de^{-(1-\e) n^\rho}/2.\end{equation}

For the same reasons as before, 
 if    $d<  \pi   \sqrt{   n \over 2  n^\rho}
$, then    $\{ 1\le j  <d/2:{ \pi j\over d}\in   A_n\}=\emptyset$. 
We obtain in a similar fashion to (1),
$$  \sup_{u\ge 0}\,\sup_{d<  (\pi/\sqrt 2) n^{(1-\rho)/2} }\big|\P\big\{  d|   B_{ n} +u\big\}-{1\over d} 
\big|\le e^{-(1-\e) n^\rho}.$$
\end{proof}


\section{Auxiliary Results}   \label{s4}


Let     $\mathcal L(v_0,D)=\big\{v_k=v_0+Dk,k\in \Z\big\}$, where $v_0 $ and $D>0$ are some reals. Let $X$ be
  a random variable   such that  $\P\{X
\in\mathcal L(v_0,D)\}=1$. 
 Put
$$ f(k)= \P\{X= v_k\}, \qq k\in \Z .$$
Let 
\begin{eqnarray*}
 \t_X =\sum_{k\in \Z}f(k)\wedge f(k+1).\end{eqnarray*}
 \vskip 2 pt 
 We assumed in   \eqref{varthetaX} that 
 \begin{equation*}  \t_X>0.
  \end{equation*}
If the span $D$ is not maximal this may be  not satisfied. Note that $\t_X<1$. 
Let $0<\t\le\t_X$. One can associate to $\t$ and $X$  a
sequence $  \{ \tau_k, k\in \Z\}$     of   non-negative reals such that
\begin{equation}\label{basber0}  \tau_{k-1}+\tau_k\le 2f(k), \qq  \qq\sum_{k\in \Z}  \tau_k =\t.
\end{equation}
For instance  $\tau_k=  \frac{\t}{\t_X} \, (f(k)\wedge f(k+1))  $ is suitable, but the {\it real} value of $\tau_k$ does not matter, that is, the first condition in \eqref{basber0} is the only requirement to make this coupling method work.  This is important to notice for the sequel.
 \vskip 3 pt   Now   define   a pair of random variables $(V,\e)$   as follows:
  \begin{eqnarray}\label{ve} \qq\qq\begin{cases} \P\{ (V,\e)=( v_k,1)\}=\tau_k,      \cr
 \P\{ (V,\e)=( v_k,0)\}=f(k) -{\tau_{k-1}+\tau_k\over
2}    .  \end{cases}\qq (\forall k\in \Z)
\end{eqnarray}
 One easily   verifies that
\begin{eqnarray}\begin{cases}\P\{ V=v_k\} &= \  f(k)+ {\tau_{k }-\tau_{k-1}\over 2} ,
\cr
 \P\{ \e=1\} &= \ \t \ =\ 1-\P\{ \e=0\}   .
\end{cases}\end{eqnarray}
\vskip 5 pt 
 Further, for any $a,b\in \C$,
  \begin{equation}\label{fc.(V, e)}\E_{(V,\e)}\, e^{aV+b\e}\,=\, \sum_{k\in \Z}\Big\{\tau_k \,e^{av_k +b }+\big(f(k)-{\tau_{k-1}+\tau_k\over
2}  \big)\, e^{av_k}\, \Big\}.
\end{equation}
 
\begin{lemma} \label{bpr} Let $L$
be a Bernoulli random variable    which is independent of  $(V,\e)$, and put  
 $Z= V+ \e DL$.
We have $Z\buildrel{\mathcal D}\over{ =}X$.
\end{lemma}
\begin{proof} (\cite{MD},\cite{W}) Plainly,
\begin{eqnarray*}\P\{Z=v_k\}&=&\P\big\{ V+\e DL=v_k, \e=1\}+ \P\big\{ V+\e DL=v_k, \e=0\} \cr
&=&{\P\{ V=v_{k-1}, \e=1\}+\P\{
V=v_k, \e=1\}\over 2} +\P\{ V=v_k, \e=0\}
\cr&=& {\tau_{k-1}+ \tau_{k }\over 2} +f(k)-{\tau_{k-1}+ \tau_{k
}\over 2}
= f(k).
\end{eqnarray*}
\end{proof}
 
 \begin{remark}\label{bep.b.0} This decomposition also applies 
 if $X$ is a Bernoulli random variable, we have $X\buildrel{\mathcal D}\over{ =}V+ \e L$. See Remark \ref{bpe.b} where $(V,\e)$ is defined. The usefulness of this decomposition is made clear in the proof of the second part of Theorem \ref{t2}.\end{remark}
\vskip 8 pt 
Let  $ X_j,j=1,\ldots,n$, be independent random variables,  each   satisfying   assumption  (\ref{basber1})
 and let $0<\t_i\le \t_{X_i}$, $i=1,\ldots, n$. Iterated  applications of Lemma \ref{bpr} allow us to
 associate to them a
sequence of independent vectors $ (V_j,\e_j, L_j) $,   $j=1,\ldots,n$  such that
 \begin{eqnarray}\label{dec0} \big\{V_j+\e_jD   L_j,j=1,\ldots,n\big\}&\buildrel{\mathcal D}\over{ =}&\big\{X_j, j=1,\ldots,n\big\}  .
\end{eqnarray}
Further the sequences $\{(V_j,\e_j),j=1,\ldots,n\}
 $   and $\{L_j, j=1,\ldots,n\}$ are independent.
For each $j=1,\ldots,n$, the law of $(V_j,\e_j)$ is defined according to (\ref{ve}) with $\t=\t_j$.  And $\{L_j, j=1,\ldots,n\}$ is  a sequence  of
independent Bernoulli random variables. Let $\E_{\!L}$, $\P_{\!L}$ (resp. $\E_{(V,\e)}$, $\P_{(V,\e)}$) stand  for the integration
symbols   and probability symbols relatively to the
$\s$-algebra generated by the sequence  $\{L_j , j=1, \ldots, n\}$   (resp. $\{(V_j,\e_j), j=1, \ldots, n\}$). Set
\begin{equation}\label{dec} S_n =\sum_{j=1}^n X_j, \qq  W_n =\sum_{j=1}^n V_j,\qq M_n=\sum_{j=1}^n  \e_jL_j,  \quad B_n=\sum_{j=1}^n
 \e_j .
\end{equation}
 \begin{lemma} \label{lemd}We have the
representation
\begin{eqnarray*} \{S_k, 1\le k\le n\}&\buildrel{\mathcal D}\over{ =}&  \{ W_k  +  DM_k, 1\le k\le n\} .
\end{eqnarray*}
And  $M_n\buildrel{\mathcal D}\over{ =}\sum_{j=1}^{B_n } L_j$.
 \end{lemma}

We also need the following technical lemma.     

  \begin{lemma} \label{lemk} Let $\t_{X_i}=\t>0$, $i=1,\ldots, n$,   and let $0<\theta\le \t<1$. For any positive integer $n$, we have $${\mathbb P}\{B_n\le \theta n\}\le \Big( {1-\t\over
1-\theta}\Big)^{n(1-\theta)}\Big( {
\t\over
\theta}\Big)^{n \theta}.$$ Let $1-\t<\rho<1$. There exists
$ 0< \theta <\t$,  $\theta=\theta(\rho,\t)$ such   that  for any positive integer $n$
 \begin{eqnarray*}  {\mathbb P}\{B_n\le \theta n\}\le \rho^n.
\end{eqnarray*} \end{lemma}

\begin{proof} 
   By Tchebycheff's inequality, for any $\lambda \ge 0$,
 \begin{eqnarray*}{\mathbb P}\{B_n\le \theta n\}&=&{\mathbb P}\{e^{-\lambda B_n} \ge e^{-\lambda\theta n}\}\le e^{-\lambda\theta n}{\mathbb E\,} e^{ \lambda B_n}=\big(e^{ \lambda\theta  }{\mathbb E\,} e^{
 \lambda  \e}\big)^n \cr &=& \Big(e^{ \lambda\theta  }\big[1-\t(1-e^{-\lambda})\big] \Big)^n.
 \end{eqnarray*}
 Put  $x=e^\lambda  $, ($x\ge 1$) and let $\p(x)=x^{ \theta  }\big[1-\t(1-x^{-1})\big]$. Then ${\mathbb P}\{B_n\le \theta n\}\le \p(x)^n $. We have
 $\p'(x)=x^{\theta-2}(x\theta(1-\t)-(1-\theta)\t)$. Thus ${\varphi}$ reaches its minimum   at the value  $x_0={(1-\theta)\t\over \theta(1-\t)}.$
 And we have  $ \p(x_0)=\psi(\theta)$, where we put
 $$ \psi(\theta)=\Big( {1-\t\over 1-\theta}\Big)^{1-\theta}\Big( { \t\over \theta}\Big)^{ \theta}, \qq 0<\theta\le \t.$$
  We note that $\psi({\vartheta} )=1$, $\lim_{\theta\to 0+} \psi(\theta)= 1-{\vartheta} $ and $\psi$  is nondecreasing ($(\log \psi)'(\theta)=\log\big( {
 \t\over 1-{\vartheta} }\big/{ \theta \over 1-\theta }\big)\ge 0$, $0<\theta\le \t$).
  Let $1-\t<\rho<1$. We
 may   select
 $ 0< \theta_{\rho,\t}<\t$ depending on $\rho,\t$ only such   that $\psi(\theta)=\rho$. This yields
   the bound
 \begin{eqnarray}{\mathbb P}\{B_n\le \theta n\}\le \rho^n.  
 \end{eqnarray}
 \end{proof}



\section{Proof of Theorem \ref{t2}.}  \label{s5}
 Let $0<\t \le \t_X$. Let also $  \{ \tau_k, k\in \Z\}$ be  a
sequence       of   non-negative reals satisfying condition \eqref{basber0}, and which   will be specified later on together with $\t$. We apply Lemma \ref{lemd} and 
   denote again $X_j= V_j+D\e_jL_j$, $S_n= W_n  +  M_n$,
$j,\, n\ge 1$. 
    \vskip 3 pt
We now  note that \begin{eqnarray}\label{dep} \P \{d|S_n+u  \}   &=& \E_{(V,\e)}   \,   \P_{\!L}
\Big\{d|\Big(  D\sum_{j= 1}^n \e_jL_j+W_n +u \Big)
\Big\}
. \end{eqnarray}
As    $\sum_{j= 1}^n \e_jL_j\buildrel{\mathcal D}\over{ =}\sum_{j=1}^{B_n } L_j$, we have
 \begin{eqnarray*}        \P_{\!L}
\Big\{d\,\big| \Big(  D\sum_{j= 1}^n \e_jL_j+W_n +u \Big)
\Big\}&=&    \P_{\!L}
\Big\{d\,\big|\Big(  D\sum_{j=1}^{B_n } L_j+W_n +u \Big)
\Big\}
\cr &=&    \P_{\!L}
\Big\{d\,\big|\Big(  D\sum_{j=1}^{B_n } L_j+W_n +u \Big)
\Big\}. 
\end{eqnarray*}
 By assumption $D=1$.   Let 
 $A_n= \big\{B_n\le \theta n \big\}.$ 
We  have
 \begin{eqnarray*}& &  \P \{d|S_n +u  \}   -  {1\over d}\sum_{ 0\le |\ell|<d }
 \E_{(V,\e)}   \, e^{i \pi(2(W_n+u)+B_n) {\ell\over   d }
 - {  B_n\pi^2\ell^2\over 2 d^2}  } 
\cr &= &  \E_{(V,\e)} 
  \, \Big( \chi(A_n)+\chi(A_n^c)\Big) 
 \, \Big\{    \P_{\!L}
\big\{d|\big(  \sum_{j= 1}^n \e_jL_j+W_n+u  \big)
\big\}
\cr & & -{1\over d}\sum_{ 0\le |\ell|<d }
     \, e^{i \pi(2(W_n+u)+B_n) {\ell\over   d }
 - {  B_n\pi^2\ell^2\over 2 d^2}  } 
\Big\}.
\end{eqnarray*}

\vskip 3 pt  

On the one hand 
by Lemma \ref{lemk},
\begin{align*}  \E_{(V,\e)} 
   \chi(A_n) 
 \, \Big|    \P_{\!L}
\big\{d|\big(  \sum_{j= 1}^n \e_jL_j+W_n +u \big)
\big\}  -& {1\over d}\sum_{ 0\le |\ell|<d }
     \, e^{i \pi(2(W_n+u)+B_n) {\ell\over   d }{  B_n\pi^2\ell^2\over 2 d^2}  } 
\Big|
 \cr   &    \qq\qq \le \,2 \P\{A_n \} \, \le  \, 2\rho^n.
\end{align*}

On the other hand,  
Theorem \ref{estPd.Bn.u} implies that for some $C$ universal,
  \begin{align}\label{basic.P}   
  \sup_{u\ge 0}\,\sup_{d\ge 2} &\Big|     \P_{\!L}
\Big\{ d\,|\,    \sum_{j=1}^{B_n } L_j+W_n   +u
\Big\}
\cr & -{1\over d}\sum_{ 0\le |\ell|<d }
 e^{i \pi(2(W_n+u)+B_n) {\ell\over   d }- {  B_n\pi^2\ell^2\over 2 d^2}  } \Big|
\,\le \, C \,\frac{(\log B_n)^{5/2}}{B_n^{3/2}} .
 \end{align}
This bound being true $\P_{(V,\e)}$-almost surely. We  observe that the function $g(x)=\frac{(\log x)^{5/2}}{x^{3/2}}$ decreases on the half-line $[e^{5/3},\infty)$.
 Thus for $  n \ge \max(n_0,e^{5/3}/\theta):=N$,  where $n_0$ arises from  Theorem \ref{estPd.Bn.u},
 \begin{align*}  \E_{(V,\e)} 
   \chi(A_n^c) 
 \, \Big|    \P_{\!L}
&\big\{d|\big(   \sum_{j= 1}^n \e_jL_j+W_n  \big)
\big\}
    -{1\over d}\sum_{ 0\le |\ell|<d }
     \, e^{i \pi(2W_n+B_n) {\ell\over   d }{  B_n\pi^2\ell^2\over 2 d^2}  } 
\Big|
\cr& \le   C \,\E_{(V,\e)} 
   \chi\big\{B_n> \theta n
\big\} \frac{(\log B_n)^{5/2}}{B_n^{3/2}}\  \le    \ C   \frac{(\log  \theta n)^{5/2}}{ (\theta n)^{3/2}}
\cr&
  \le \,  \frac{C }{ \theta ^{3/2}}\ \frac{(\log    n)^{5/2}}{ n^{3/2}} \,,
\end{align*}
uniformly over $u\ge 0$, $d\ge 2$.
Therefore, 
 \begin{align}\label{ddivSniid}    \sup_{u\ge 0}\,\sup_{d\ge 2}\Big|\, \P \{d|S_n+u  \}   -  {1\over d}\sum_{ 0\le |\ell|<d }
 \E_{(V,\e)}    \, &e^{i \pi(2(W_n+u)+B_n)  {\ell\over   d }  
 - {  B_n\pi^2\ell^2\over 2 d^2}} \,\Big|
 \cr &  \,\le  \, \frac{C }{ \theta ^{3/2}}\   \frac{(\log  n)^{5/2}}{ n^{3/2}}+2\rho^n.
\end{align}
\vskip 4 pt Now, as $W_n =\sum_{j=1}^n V_j$,  $B_n=\sum_{j=1}^n
 \e_j$,
\begin{eqnarray*}    {1\over d}\sum_{ 0\le |\ell|<d }
 e^{i \pi(2(W_n+u)+B_n){\ell\over   d }-{  B_n \pi^2\ell^2\over 2 d^2}}
&=&  {1\over d}\sum_{ 0\le |\ell|<d }
e^{i \pi (2(\sum_{j=1}^n V_j+u)+\sum_{j=1}^n\e_j ){\ell\over   d }
-( \sum_{j=1}^n \e_j){ \pi^2\ell^2\over 2 d^2}} 
 \cr &=&  {1\over d}\sum_{ 0\le |\ell|<d }e^{ i \pi (2\sum_{j=1}^n V_j+u){\ell\over   d }+(\sum_{j=1}^n
 \e_j)(i\pi {\ell\over   d }-{ \pi^2\ell^2\over 2 d^2})}
   . 
 \end{eqnarray*}
 By integrating,
  \begin{align*}     {1\over d}\sum_{ 0\le |\ell|<d }\E_{(V,\e)}\,&e^{ i \pi (2\sum_{j=1}^n V_j+u){\ell\over   d }+(\sum_{j=1}^n
 \e_j)(i\pi {\ell\over   d }-{ \pi^2\ell^2\over 2 d^2})}
 \cr &\ =  \, {1\over d}\sum_{ 0\le |\ell|<d } \Big(\E_{(V,\e)}\,e^{ i \pi ( 2V +u){\ell\over   d }+
 \e  (i\pi {\ell\over   d }-{ \pi^2\ell^2\over 2 d^2})}\Big)^n
. 
 \end{align*}
Recalling   \eqref{fc.(V, e)},
 \begin{eqnarray*}\E_{(V,\e)}\, e^{aV+b\e}\,=\, \sum_{k\in \Z}\Big\{\tau_k \,e^{av_k +b }+\big(f(k)-{\tau_{k-1}+\tau_k\over
2}  \,\big)\, e^{av_k}\, \Big\}
\cr ,\end{eqnarray*} 
it follows that 
\begin{align*}
\label{ }\E_{(V,\e)}&\, \,e^{ i \pi ( 2V +u){\ell\over   d }+
 \e  (i\pi {\ell\over   d }-{ \pi^2\ell^2\over 2 d^2})}
\cr&  =\, \sum_{k\in \Z}\Big(\tau_k \,e^{i \pi {\ell\over   d }(2v_k +u)+(i\pi {\ell\over   d }-{ \pi^2\ell^2\over 2 d^2}) }+ \big(f(k)-{\tau_{k-1}+\tau_k\over
2}  \,\big)  \, e^{ i \pi ( 2v_k +u)  {\ell\over   d }}\, \Big).
 \end{align*} 
 
Hence we are left with the sum
\begin{eqnarray} 
\label{sum1}  
   & &{1\over d} \sum_{ 0\le |\ell|<d }\E_{(V,\e)}\,
   e^{ i \pi (2\sum_{j=1}^n V_j+u){\ell\over   d }+
   (\sum_{j=1}^n \e_j)(i\pi {\ell\over   d }-{ \pi^2\ell^2\over 2 d^2})} 
 \cr &  = & {1\over d}\sum_{ 0\le |\ell|<d } \Big(\E_{(V,\e)}\,e^{ i \pi ( 2V +u){\ell\over   d }+
 \e  (i\pi {\ell\over   d }-{ \pi^2\ell^2\over 2 d^2})}\Big)^n
  \cr &  = & {1\over d}\sum_{ 0\le |\ell|<d } \bigg(\sum_{k\in \Z}\Big(\tau_k \,e^{i \pi {\ell\over   d }(2v_k +u)+(i\pi {\ell\over   d }-{ \pi^2\ell^2\over 2 d^2}) }+ \big(f(k)-{\tau_{k-1}+\tau_k\over
2}  \,\big)  \, e^{ i \pi ( 2v_k +u)  {\ell\over   d }}\, \Big)\bigg)^n.
\end{eqnarray}
   \vskip 9 pt 

We now  specify      the sequence $  \{ \tau_k, k\in \Z\}$ and hence $\t$.
     We choose $\t$ according to \eqref{basber1} so that by \eqref{tau.mu} the equation   
 \begin{equation}
\label{tau.mu1}
 {\tau_{k-1}+\tau_k\over 2} =f(k)  - \m_k,
\end{equation}
 is satisfied for all $k \in \Z$, recalling that $  \{ \m_k, k\in \Z\}$ is a sequence of non-negative  reals satisfying condition  \eqref{muk}, $\m= \sum_{k\in \Z}\m_k$ satisfies condition \eqref{mu}, namely $1-\m<\t_X$.
 
\vskip 2 pt The  solutions to Equation \eqref{tau.mu}    are 
\begin{eqnarray}\label{tau.mu1.sol}\qq \tau_m= \begin{cases}  \sum_{\ell\le j} (x_{2\ell}-x_{2\ell-1}) \ &\quad \hbox{ if   $m=2j$},
\cr   \sum_{\ell\le j} (x_{2\ell+1}-x_{2\ell}) \ &\quad \hbox{   $m=2j+1$},
\end{cases}
\end{eqnarray}
 where we set $x_u=2 (f(u) - \m_u )$. 
So that,
$$ \tau_{2j}+\tau_{2j+1}=\sum_{\ell\le j} (x_{2\ell}-x_{2\ell-1})+\sum_{\ell\le j} (x_{2\ell+1}-x_{2\ell})=x_{2j+1}=2 (f(2j+1) - \m_{2j+1} ).$$
Quite similarly,
$$ \tau_{2j}+\tau_{2j-1} =x_{2j}=2 (f(2j) - \m_{2j} ).$$
So that \eqref{tau.mu1} holds.
 By construction 
\begin{equation}\label{tau.f.theta}
  \tau_{k-1}+\tau_k\le 2f(k), \qq  \qq \t:= \sum_{k\in \Z}  \tau_k = 1- \m.
\end{equation}
Thus   \eqref{basber0}  is obviously satisfied.
Note that quantities $\tau_k$, $\m_k$ do not rely on $n$. 
\vskip 5 pt 
  Using \eqref{widetildeX},
\eqref{sum1} may be continued as follows,  
\begin{eqnarray} \label{sum2}
& &\, 
\cr &=& {1\over d}\sum_{ 0\le |\ell|<d } \bigg(\sum_{k\in \Z}\Big(\tau_k \,e^{i \pi {\ell\over   d }(2v_k +u)+(i\pi {\ell\over   d }-{ \pi^2\ell^2\over 2 d^2}) }
\cr & &\qquad \qquad \qquad \qquad \qquad + \big(f(k)-{\tau_{k-1}+\tau_k\over
2}  \,\big)  \, e^{ i \pi ( 2v_k +u)  {\ell\over   d }}\, \Big)\bigg)^n
\cr &=& {1\over d}\sum_{ 0\le |\ell|<d } \bigg(\sum_{k\in \Z}\Big(\tau_k \,e^{i \pi {\ell\over   d }(2v_k +u)+(i\pi {\ell\over   d }-{ \pi^2\ell^2\over 2 d^2}) }+ \m_k  \, e^{ i \pi ( 2v_k +u)  {\ell\over   d }}\, \Big)\bigg)^n
\cr  
 &=&  {1\over d}\sum_{ 0\le |\ell|<d } \Big( e^{ (i\pi {\ell\over   d }-{ \pi^2\ell^2\over 2 d^2}) } \t\,
 \E \,e^{i \pi  {\ell\over   d }( 2\widetilde X +u)   }   +s\big( \pi ( 2v_k +u)  {\ell\over   d }\big)\Big)^n.
\end{eqnarray}
By carrying back estimate \eqref{sum2}   to \eqref{ddivSniid}, we finally get in view of \eqref{sum1},
 \begin{align}\label{ddiv.final}
  \Big| \P \{d|S_n +u \}   -  {1\over d}\sum_{ 0\le |\ell|<d } & \Big( e^{ (i\pi {\ell\over   d }-{ \pi^2\ell^2\over 2 d^2}) } \t\,
 \E \,e^{i \pi  {\ell\over   d }( 2\widetilde X +u)   }   +s\big( \pi ( 2v_k +u)  {\ell\over   d }\big)\Big)^n \Big|
  \cr  &\le    \frac{C }{ \theta ^{3/2}}\   \frac{(\log  n)^{5/2}}{ n^{3/2}}+2\rho^n.
\end{align}
Here $C$  is universal, and this is true for all    $u\ge 0$, $d\ge 2$, $n\ge 2$. The proof of the first assertion is  achieved by taking   in both sides the supremum over all $u\ge 0$ and $d\ge 2$, and noticing that   in the inner sum of the left-term, the summand  corresponding to $\ell = 0$ is equal to  $\t+\m=1$.
\vskip 8 pt 
 We now prove the second assertion. An obvious consequence of  the first assertion is  that for $d\ge 2$, $n\ge 2$, 
   \begin{align}\label{ddiv.final.conseq.} \sup_{u\ge 0} \, &\Big| \P \{d|S_n +u \}   - {1\over d}\Big|
   \cr & \,\le \, {1\over d}\sum_{ 0< |\ell|<d }  \big( e^{ -{ \pi^2\ell^2\over 2 d^2} } \t\,
    +\m \big)^n  
  +    \frac{C }{ \theta ^{3/2}}\   \frac{(\log  n)^{5/2}}{ n^{3/2}}+2\rho^n.
\end{align}
 On expanding the summand and recalling that  $\m=1-\t$, we get    
   \ \begin{align}\label{ddiv.final.conseq.1}     \sup_{u\ge 0} \, \Big| \P \{d|S_n+u  \}   - {1\over d}\Big|&
   \,\le \, {1\over d}\,\sum_{ 0< |\ell|<d }\   \sum_{m=0}^n {n\choose m} \t^m\,e^{  -m{ \pi^2\ell^2\over 2 d^2}  }  (1-\t)^{n-m}
     \cr  &  \quad +    \frac{C }{ \theta ^{3/2}}\   \frac{(\log  n)^{5/2}}{ n^{3/2}}+2\rho^n
    .
\end{align}
 {Let $\mathcal D$ be a test set of divisors $\ge 2$, and note $\mathcal D_\p$ the section of $\mathcal D$ at height $\p$, and $|\mathcal D_\p|$   its   cardinality. Let $0<\e<1$.   Then for all reals $u\ge 0$, all integers $\p\ge 2$,
 \begin{eqnarray}\label{ddiv.final.conseq.2}& &   \,{1\over |\mathcal D_\p |}  \sum_{d\in \mathcal D_\p }\,\Big| \P \{d|S_n+u  \}   - {1\over d}\Big|
    \cr &      \le & {1\over |\mathcal D_\p |} \sum_{d\in \mathcal D_\p }    \,{1\over d}\sum_{ 0< |\ell|<d }  \,  \sum_{m=0}^n {n\choose m} \t^m(1-\t)^{n-m}
   \,   e^{  -m{ \pi^2\ell^2\over 2 d^2}  }  
      +    \frac{C }{ \theta ^{3/2}}\   \frac{(\log  n)^{5/2}}{ n^{3/2}}+2\rho^n
\cr &      \le & {1\over |\mathcal D_\p |} \sum_{d\in \mathcal D_\p }    \,{1\over d}\sum_{ 0< |\ell|<d }  \,  \sum_{m=0}^n {n\choose m} \t^m\big(1-(e^{  -\e { \pi^2\ell^2\over 2 d^2}  }\t)\big)^{n-m}
   \,   e^{  -m{ \pi^2\ell^2\over 2 d^2}  }  
\cr & &      +    \frac{C }{ \theta ^{3/2}}\   \frac{(\log  n)^{5/2}}{ n^{3/2}}+2\rho^n  ,\end{eqnarray}
 where we have bounded in the last line of calculations $(1-\t)$ by $\big(1-(e^{  -\e { \pi^2\ell^2\over 2 d^2}  }\t)\big)$. This device will permit us to well estimate  the above  sums.\vskip 5 pt  
 Then,
 \begin{align}\label{est.1} \sum_{n=2}^\infty \, &\sup_{u\ge 0} \, \sup_{\p\ge 2}\,  {1\over |\mathcal D_\p |} \sum_{d\in \mathcal D_\p } \,\Big| \P \{d|S_n+u  \}   - {1\over d}\Big|
  \cr   &\,\le  \sup_{u\ge 0} \,\sup_{\p\ge 2}\,{1\over |\mathcal D_\p |} \sum_{d\in \mathcal D_\p }\sum_{ 0< |\ell|<d }   \,{1\over d}\,\sum_{n=2}^\infty \ \sum_{m=0}^n {n\choose m}  e^{  - m{ \pi^2\ell^2\over 2 d^2}  }         
 \t^m  \big(1-(e^{  -\e { \pi^2\ell^2/2 d^2}  }\t)\big)^{n-m} 
   +H,
\end{align}  
where we set   
$$H=    \frac{C }{ \theta ^{3/2}}\  \sum_{n=2}^\infty  \frac{(\log {n})^{5/2} }{ ({n})^{3/2}}+2\sum_{n=2}^\infty  \rho^n.$$

Now we utilize a second   device. By permuting sums, 
\begin{align} &\,\le   \sup_{u\ge 0} \,\sup_{\p\ge 2}\,\,{1\over |\mathcal D_\p |}  \sum_{d\in \mathcal D_\p }\sum_{ 0< |\ell|<d }   \,{1\over d} \,\sum_{m=0}^\infty e^{  -  m{ \pi^2\ell^2\over 2 d^2}  } \t^m          \Big\{\,\sum_{n\ge m } {n\choose m} \big(1-(e^{  -\e { \pi^2\ell^2\over 2 d^2}  }\t)\big)^{n-m}    \Big\}
+ H.
\end{align}
 Let $\t_1= e^{  -\e { \pi^2\ell^2/ 2 d^2}  }\t$. The   sum in brackets   $\sum_{n\ge m} {n\choose m}(1-\t_1)^{n-m}    
$ is $\t_1^{-m-1}$. This follows from the formula $\sum_{v=0}^\infty {v+z\choose z}
\, x^v={1\over
(1-x)^{z+1}} $ valid for $|x|<1$,  $z\ge 0$. 
\vskip 3 pt 
We can thus continue as follows  
\begin{eqnarray}
   &=&  \sup_{u\ge 0} \,\sup_{\p\ge 2}\,  {1\over |\mathcal D_\p |}  \sum_{d\in \mathcal D_\p }\sum_{ 0< |\ell|<d }  \,{1\over d} \,\sum_{m=0}^\infty e^{  -  m{ \pi^2\ell^2\over 2 d^2}  } \t^m        
   (e^{  -\e { \pi^2\ell^2\over 2 d^2}  }\t)^{-m-1}
+H 
 \cr &=&  \sup_{u\ge 0} \,\sup_{\p\ge 2}\, {1\over |\mathcal D_\p |}  \sum_{d\in \mathcal D_\p }\sum_{ 0< |\ell|<d }   \,{e^{   \e { \pi^2\ell^2\over 2 d^2}  }\over \t\,d} \,\sum_{m=0}^\infty e^{   -(1-\e)  m{ \pi^2\ell^2\over 2 d^2}  }         
     +H
     \cr &=&   \sup_{u\ge 0} \,\sup_{\p\ge 2}\, {1\over \t\,|\mathcal D_\p |}  \sum_{d\in \mathcal D_\p }\sum_{ 0< |\ell|<d }   \, 
         \frac{e^{   \e { \pi^2\ell^2/ 2 d^2}  }}
     {d(1-e^{   -(1-\e)   { \pi^2\ell^2/ 2 d^2} }   ) }       +H
   \cr &\le &   \sup_{u\ge 0} \,\sup_{\p\ge 2}\, {1\over \t\,|\mathcal D_\p |}  \sum_{d\in \mathcal D_\p }\sum_{ 0< |\ell|<d }   \, 
         \frac{e^{   \e { \pi^2/ 2 }  }}
     {d(1-e^{   -    { (1-\e)\pi^2/ 2 d^2} }   ) }       +H
      \cr &\le &    
       \frac{2e^{   \e { \pi^2/ 2 }  }}
     {\t\, (1-e^{   -    { (1-\e)\pi^2/ 8} }   ) }
 \, 
                +H.  
      \end{eqnarray} 
By replacing $H$ by its value we obtain   the following bound, 
 \begin{align}\label{est.1.} \sum_{n=2}^\infty \,  \sup_{u\ge 0} \,\sup_{\p\ge 2}\,&{1\over |\mathcal D_\p |} \sum_{d\in \mathcal D_\p } \,\Big| \P \{d|S_n+u  \}   - {1\over d}\Big|
  \cr   \le  &  \  \frac{2e^{   \e { \pi^2/ 2 }  }}
     {\t\, (1-e^{   -    { (1-\e)\pi^2/ 8} }   ) }
 \,  
     \   +  \frac{C }{ \theta ^{3/2}}\  \sum_{n=2}^\infty  \frac{(\log {n})^{5/2} }{ ({n})^{3/2}}+2\sum_{n=2}^\infty  \rho^n.
\end{align}
From there by taking $\e= 1/2$ assertion (ii) easily follows.


}

\begin{remark}\label{bpe.b} If $X$ is a Bernoulli random variable, we have a similar decomposition. First $D=1$, $v_k=k$, $k\in \Z$, $f(k)=1/2$ if $k\in\{ 0,1\}$ and is 0 otherwise.  Whence $f(k)\wedge f(k+1)=1/2$ if $k=0$, and   0 otherwise, so that  $\t_X=1/2$. 
The  first condition in \eqref{basber0} further implies that $\tau_k=0$   if $k\neq0$, and   $\tau_0\le 1$. We have $\t=\tau_0$,  where $\t$ can be chosen in $[0, 1/2]$. 
This is a case where the terms of  the sequence $  \{ \tau_k, k\in \Z\}$ are all   0 except $\tau_0$. The pair $(V,\e)$ is defined as follows 

\begin{eqnarray}\label{ve.bernoulli}
 \qq\qq\begin{cases} \P\{ (V,\e)=( 0,1)\}=\tau_0,      \cr
 \P\{ (V,\e)=( 0,0)\}=1/2 -{ \tau_0\over
2}
   \cr
 \P\{ (V,\e)=( 1,0)\}=1/2 -{\tau_{0} \over
2}    .  \end{cases}\end{eqnarray}
Thus $\P\{ V= 0\}=1/2 +\tau_0/2$, $\P\{ V= 1\}=1/2-\tau_0/2$, $\P\{ V= k\}=0$ otherwise, $\P\{ \e=1\}  =  \t   =\ 1-\P\{ \e=0\}$. By Lemma \ref{bpr} $X=  V+ \e  L$. Furthermore 
$$ \E_{(V,\e)}\, e^{aV+b\e}\,=\,  \tau_0 \,e^{b }+\big(1/2-{ \tau_0\over
2}  \big)+\big(1/2-{ \tau_0\over
2}  \big)\,e^{a} .$$
  \end{remark}

\begin{remark}
We claim that   
\begin{eqnarray}\label{sup.phi}
\sup_{2\le \p\le n} \, \frac{1}{\p}\sum_{2\le d<\p}{1\over d}\  \sum_{ 1\le \ell <d }\, e^{  -m{ \pi^2\ell^2\over 2 d^2}  } &\le &     \frac{C}{\sqrt m}.  \end{eqnarray} 
  Let $A $ and   $U$ be positive {\it reals}. Then for $d\ge 1$,
$$ { 1\over 2d}e^{-A({U\over
d})^2}\le \int_{U\over d+1}^{U\over d}e^{-At^2}{  d t\over t}.  $$Indeed,
$$ \int_{U\over d+1}^{U\over d}e^{-At^2}{  d t\over t}\ge e^{-A({U\over d})^2}\int_{U\over d+1}^{U\over d} {  d t\over t}=e^{-A({U\over
d})^2}\log (1+{ 1\over d})\ge  { 1\over 2d}e^{-A({U\over
d})^2} , $$
since $\log 1+x\ge \frac{x}{2}$, if $0\le x\le 1$. Apply this with the choices $U=\sqrt m \ell$, $A= \frac{\pi^2}{2}$. We get
$$ { 1\over 2d}e^{-\frac{\pi^2}{2}({\sqrt m \ell\over
d})^2}\le \int_{{\sqrt m \ell\over d+1}}^{{\sqrt m \ell\over d}}e^{-\frac{\pi^2}{2}t^2}{  d t\over t}.$$
Thus 
\begin{eqnarray*}
& &\sum_{2\le d<\p}{1\over d}\  \sum_{ 1\le \ell <d }\, e^{  -m{ \pi^2\ell^2\over 2 d^2}  } 
\, \le \, 2 \sum_{2\le d<\p} \  \sum_{ 0< \ell <d }\int_{{\sqrt m \ell \over d+1}}^{{\sqrt m \ell\over d}}e^{-\frac{\pi^2}{2}t^2}{  d t\over t}
\cr &= & 2\sum_{ 1\le \ell <\p }\sum_{\ell \le d<\p} \ \int_{{\sqrt m \ell\over d+1}}^{{\sqrt m \ell\over d}}e^{-\frac{\pi^2}{2}t^2}{  d t\over t}
\,=\,2\sum_{ 1\le \ell <\p }\ \int^{\sqrt m }_{{\sqrt m \ell \over \p}}   e^{-\frac{\pi^2}{2}t^2}{  d t\over t}
\cr &\le  & 2 \ \int^{\sqrt m }_{\sqrt m  \over \p} \Big(\sum_{ 1\le \ell\le  \min(\p, \frac{\p t}{\sqrt m }}1\Big)  e^{-\frac{\pi^2}{2}t^2}{  d t\over t}
\cr &\le  & 2 \  \frac{\p}{\sqrt m}\,\int^{\sqrt m }_{{\sqrt m  \over \p}}    e^{-\frac{\pi^2}{2}t^2}   d t 
\, \le \,  2 \  \frac{\p}{\sqrt m}\,\int^{\infty }_{0}    e^{-\frac{\pi^2}{2}t^2}   d t
\, = \,   C\,  \frac{\p}{\sqrt m}.  \end{eqnarray*}
This is true whatsoever $\p$ such as $2\le \p\le n$. Therefore, 
\begin{eqnarray*}
\sup_{2\le \p\le n} \, \frac{1}{\p}\sum_{2\le d<\p}{1\over d}\  \sum_{ 1\le \ell <d }\, e^{  -m{ \pi^2\ell^2\over 2 d^2}  } &\le &    \frac{C}{\sqrt m}, \end{eqnarray*}
which is \eqref{sup.phi}. 
\end{remark}

\vskip 8 pt 
  
 \begin{remark}
  It concerns the role of the summand 
$\big( e^{ -{ \pi^2\ell^2\over  2 d^2} } \t\,
    +\m \big)^n$. Let $p>1$. 
    By H\"older's inequality, recalling that $\m=1-\t$,
\begin{eqnarray*} \big( e^{  -{ \pi^2\ell^2\over 2 d^2}  } \t\,
   +\m \big)^{np} &=&\Big( \sum_{m=0}^n {n\choose m} \t^m(1-\t)^{n-m}e^{  -m { \pi^2\ell^2\over 2 d^2}  }  \Big)^{ p}
   \cr & \le &\sum_{m=0}^n {n\choose m} \t^m(1-\t)^{n-m}e^{  -mp{ \pi^2\ell^2\over 2 d^2}  }
   \cr &= & \big( e^{  -{ \pi^2p\ell^2\over 2 d^2}  } \t\,
   +\m \big)^{n } .
   \end{eqnarray*}
 So that if $p $ is  an integer, the summand corresponding to $S_{np}$ expresses as the summand corresponding to $S_n$, altered by a factor $p$ in the exponent term.
 \end{remark}

 {

\end{document}